\documentclass[]{interact}

\usepackage{graphicx}
\usepackage{graphicx}
\usepackage{amsfonts}
\usepackage{epstopdf}
\usepackage{algorithmic}
\usepackage[english]{babel}
\usepackage{amsmath}
\usepackage{mathptmx}
\usepackage{amsmath}
\usepackage{color}
\usepackage{amssymb}
\usepackage{latexsym}
\usepackage{helvet}
\usepackage{epstopdf}
\usepackage{wasysym}
\usepackage{pgf}
\usepackage{comment}
\usepackage{mathptmx}       
\usepackage{helvet}         
\usepackage{multicol}        
\usepackage[bottom]{footmisc}
\usepackage{lineno}

\usepackage{epstopdf}
\usepackage[caption=false]{subfig}

\usepackage[numbers,sort&compress]{natbib}
\bibpunct[, ]{[}{]}{,}{n}{,}{,}
\renewcommand\bibfont{\fontsize{10}{12}\selectfont}
\makeatletter
\def\NAT@def@citea{\def\@citea{\NAT@separator}}
\makeatother

\theoremstyle{plain}
\newtheorem{theorem}{Theorem}[section]

\theoremstyle{definition}
\newtheorem{definition}[theorem]{Definition}

\theoremstyle{remark}
\newtheorem{remark}{Remark}

\begin{document}

\articletype{Research Article}

\title{A stochastic control approach for constrained stochastic differential games with jumps and regimes}

\author{
\name{E. Savku\thanks{CONTACT E. Savku. Email: esavku@gmail.com} }
\affil{University of Oslo, Department of Mathematics, Oslo, Norway}
}

\maketitle

\begin{abstract}
We develop an approach for two player constraint zero-sum and nonzero-sum stochastic differential games, which are modeled by Markov regime-switching jump-diffusion processes. We provide the relations between a usual stochastic optimal control setting and a Lagrangian method. In this context, we prove corresponding theorems for two different type of constraints, which lead us to find real valued and stochastic Lagrange multipliers, respectively. Then, we illustrate our results for a nonzero-sum game problem with stochastic maximum principle technique. Our application is an example of cooperation between a bank and an insurance company, which is a popular, well-known business agreement type, called Bancassurance. 
\end{abstract}

\begin{keywords}
Stochastic optimal control, Stochastic Differential Games, Regime-Switches, Stochastic Maximum Principle, Insurance.
\end{keywords}

\section{Introduction}
\label{sec:intro}
A regime-switching model is one of the most powerful tools to capture abrupt changes efficiently in a wide range of random phenomenons. The discrete shifts from one state to another may easily describe financial, natural or mechanical events mathematically; hence, they enjoy with a substantial application area. \\

\noindent In this work, specifically, we focus on the one in the field of finance and actuarial science. The states of a Markov chain can be seen as the proxies of the macroeconomic instruments such as gross domestic product or sovereign credit rating. Furthermore, when we observe how regulation policies issued by governments or financial institutions lead deep modifications in the microstructure of the financial markets (see \cite{savku4}), the importance of regime-switching models appears brightly. Moreover, the periods arised after more catastrophic events like a financial crisis; e.g. the bankruptcy of Lehman Brothers in 2008, can be described by such systems. Additionally, we can combine regime-switches with stochastic optimal control, which is another fundamental method of managing random events (for complete treatments of control theory, see \cite{oksendal3} and \cite{yong}). Hence, these models have attracted many researchers so far such as \cite{crepey2, elliott1, zes:02, pamen, lv, mao, savku2, savku3, savku}. \\

\noindent In this work, we also utilize the foundations of stochastic differential games and take stochastic optimal control and regime-switches together. Many authors have paid attention to this combination as well, see \cite{elliott1, elliott3, savku, ma, shen, zhang, deepa, bui} and references therein. Such problems can be solved with the methods of both dynamic programing principle and stochastic maximum principle. Particularly, we focus on an application of nonzero-sum stochastic differential game, for which stochastic maximum principle has been preferred as a solution technique. In this sense, we want to mention \cite{pamen, zhang} and \cite{deepa}.\\

\noindent In \cite{pamen}, authors develop  necessary and sufficient maximum principles for a Markov regime-switching forward-backward zero-sum and nonzero-sum stochastic differential games. Then, they provide an application for zero-sum game formulation, which describes a robust utility maximization under a relative entropy penalty and find optimal investment of an insurance firm under model uncertainty. But they do not give an example for nonzero-sum game formulation. In \cite{zhang}, authors investigate optimal dividend strategies for two insurance companies and model their work with a stopping time problem via a regime-switching diffusion process and provide a verification theorem. In \cite{deepa}, authors study on an optimal control problem of nonzero-sum game mean-field delayed Markov regime-switching forward-backward stochastic system with Lévy processes associated with Teugels martingales over an infinite time horizon. In this context, they provide necessary and sufficient maximum principles of these type problems. Moreover, they define a single state process of the system for both of the players and try to maximize cost functionals of each player. \\

\noindent In our work, we extend Theorem 11.3.1 in \cite{oksendal2}, which has been proved for stochastic optimal control problems without regime-swithches by the authors. Later, \cite{dahl1} have applied this theorem for an application of the effects of inflation and wage risk on optimal consumption via stochastic maximum principle. Hence, our main contributions are to develop corresponding theorems for a zero-sum and a nonzero-sum stochastic differential game formulations and generalizing state processes of the system to a Markov regime-switching jump-diffusion environment. Our theorems can be applied with both of the dynamic programming principle and stochastic maximum principle methods as well.\\

\noindent This paper is organized as follows: In Section \ref{sec:2}, we provide the details of the model dynamics. Then, we introduce our Markov regime-swithcing jump-diffusion process, which is going to correspond to the state process of the system in our game theoretical application. In Section \ref{sec:3}, we extend Theorem 11.3.1 in \cite{oksendal2} to develop techniques in oerder to find the saddle point of a zero-sum game. In Section \ref{sec:4}, we generalize Theorem 11.3.1 in \cite{oksendal2} for  the Nash equilibrium concept, which presents the stochastic optimal control processes of a nonzero-sum game formulation. In Section \ref{apply}, we investigate a cooperation between a bank and an insurance company via a nonzero-sum stochastic differential game method. While the company makes a decision for the optimal dividend payment against the best decision of the bank, the bank tries to determine optimal appreciation rate for its cash flow corresponding to best action of the company, and vice versa. In Section \ref{conc}, we provide an insight about our results. Finally, a version of sufficient maximum principle theorem with all required technical conditions can be found in Apendix.

\section{Preliminaries}
\label{sec:2}

Throughout this work, we assume that the maturity time $T>0$ is finite. Let $(\Omega,\mathbb{F},\left(\mathcal{F}_{t}\right)_{t\geq 0},\mathbb{P})$ be a complete probability space, where $\left(\mathcal{F}_{t}\right)_{t\geq 0}$ is a right-continuous, $\mathbb{P}$-completed filtration and $\mathbb{F}=\left(\mathcal{F}_{t}:t\in [0,T]\right)$ is generated by an $M$-dimensional Brownian motion $W(\cdot)$, an $L$-dimensional Poisson random measure $N(\cdot,\cdot)$ and a $D$-state Markov chain $\alpha(\cdot)$. It is assumed that these processes are independent of each other and adapted to $\mathbb{F}$. \\
Let $\left(\alpha(t): t\in [0,T]\right)$ be a continuous-time, finite-state Markov chain. We can choose a time-homogenous or a time-inhomogenous Markov chain depending on the application that we purpose to formulate. Moreover, based on the specific problem, the chain may be reducible or irreducible.\\
Let us represent the canonical state space of the finite-state Markov chain $\alpha(t)$ by $S=\left\{e_{1},e_{2},...,e_{D}\right\}$, where $D\in \mathbb{N}$, $e_{i}\in\mathbb{R}^{D}$ and the $j$th component of $e_{i}$ is the Kronecker delta $\delta_{ij}$ for each pair of $i,j=1,2,...,D$. \\
The generator of the chain under $\mathbb{P}$ is defined by $\Lambda:=[\mu_{ij}(t)]_{i,j=1,2,...,D}, \ t\in [0,T]$. For each $i,j=1,2,...,D$, $\mu_{ij}(t)$ is the transition intensity of the chain from each state $e_{i}$ to state $e_{j}$ at time $t$. For $i\neq j$, $\mu_{ij}(t)\geq 0$ and $\sum_{j=1}^{D}\mu_{ij}(t)=0$; hence, $\mu_{ii}(t)\leq 0$. \\
By \cite{eam:13}-Appendix B, we know that there is a semimartingale representation for a Markov chain $\alpha$ as follows:
\begin{equation*}
\label{eq:mart}
\alpha(t)=\alpha(0)+\int_{0}^{t}\Lambda^{T}\alpha(u)du+M(t), 
\end{equation*}
where $\left(M(t):t\in [0,T]\right)$ is an $\mathbb{R}^{D}$-valued $(\mathbb{F},\mathbb{P})$-martingale and $\Lambda^{T}$ describes the transpose of the matrix. \\
In this work, we utilize a set of Markov jump martingales associated with the chain $\alpha$ as developed in \cite{zes:02}: \\
Let $J^{ij}(t)$ represent the number of the jumps from state $i$ to state $j$ up to and included time $t$ for each $i,j=1,2,...,D$, with $i\neq j$ and $t\in [0,T]$. Then,
\begin{align*}
J^{ij}(t)&:=\sum \limits_{0<s\leq t}\left\langle \alpha(s-),e_{i}\right\rangle \left\langle \alpha(s),e_{j}\right\rangle  \\
&=\sum \limits_{0<s\leq t}\left\langle \alpha(s-),e_{i}\right\rangle \left\langle \alpha(s)-\alpha(s-),e_{j}\right\rangle \\ 
&=\int_{0}^{t}\left\langle \alpha(s-),e_{i}\right\rangle \left\langle d\alpha(s),e_{j}\right\rangle  \\
&=\int_{0}^{t}\left\langle \alpha(s-),e_{i}\right\rangle \left\langle \Lambda^{T}\alpha(s),e_{i}\right\rangle ds+\int_{0}^{t}\left\langle \alpha(s-),e_{i}\right\rangle \left\langle dM(s),e_{j}\right\rangle \\
&=\int_{0}^{t}\mu_{ij}(s)\left\langle \alpha(s-),e_{i}\right\rangle ds+m_{ij}(t),
\end{align*}
where the processes $m_{ij}$'s are $(\mathbb{F},\mathbb{P})$-martingales and called the basic martingales associated with the chain $\alpha$. For each fixed $j=1,2,...,D$, let $\Phi_{j}$ be the number of the jumps into state $e_{j}$ up to time $t$. Then,
\begin{align*}
\Phi_{j}(t)&:=\sum \limits_{i=1,i\neq j}^{D}J^{ij}(t) \\
&=\sum \limits_{i=1,i\neq j}^{D} \int_{0}^{t}\mu_{ij}(s)\left\langle \alpha(s-),e_{i}\right\rangle ds+ \tilde{\Phi}_{j}(t).
\end{align*}
Let us define $\tilde{\Phi}_{j}(t):=\sum \limits_{i=1,i\neq j}^{D}m_{ij}(t)$ and $\mu_{j}(t):=\sum \limits_{i=1,i\neq j}^{D}\int_{0}^{t}\mu_{ij}(s)\left\langle \alpha(s-),e_{i}\right\rangle ds$; then for each $j=1,2,...,D$,
\begin{equation*}
\tilde{\Phi}_{j}(t)=\Phi_{j}(t)-\mu_{j}(t)
\end{equation*}
is an $(\mathbb{F},\mathbb{P})$-martingale. By $\tilde{\Phi}(t)=(\tilde{\Phi}_{1}(t),\tilde{\Phi}_{2}(t),...,\tilde{\Phi}_{D}(t))^{T}$, we represent a compensated random measure on $([0,T]\times S,\mathcal{B}([0,T])\otimes \mathcal{B}_{S})$, where $\mathcal{B}_{S}$ is a $\sigma$-field of $S$.\\
Moreover, another description of such a martingale representation for a random measure generated by a Markov chain can be found in \cite{asmus}-Appendix A.3 within the framework of actuarial science.\\
Furthermore, let $\mathcal{B}_{0}$ be the Borel $\sigma$-field generated by an open subset of $\mathbb{R}_{0}:=\mathbb{R}\setminus \left\{0\right\}$, whose closure does not contain the point $0$. We define
\begin{equation*}
\tilde{N}_{i}(dt,dz):=N_{i}(dt,dz)-\nu^{i}(dz)dt, \qquad i=1,2,\ldots,L,
\end{equation*}
which are compensated Poisson random measures and $(N_{i}(dt,dz):t\in [0,T],z\in\mathbb{R}_{0})$'s are independent Poisson random measures on $([0,T]\times \mathbb{R}_{0},\mathcal{B}([0,T])\otimes \mathcal{B}_{0})$ and $\nu^{i}(dz)=(\nu_{e_{1}}^{i}(dz),\nu_{e_{2}}^{i}(dz),\ldots,\nu_{e_{D}}^{i}(dz))^{T}$'s are L\'evy densities of jump sizes of the random measure $N_{i}(dt,dz)$ \ for $i=1,2,\ldots,D$.\\
Now, let us describe the state process of the system as a Markov regime-switching jump-diffusion process:
\begin{align}
Y(t)=&\ b(t,Y(t),\alpha(t),u_{1}(t),u_{2}(t))dt  \nonumber\\
&+\sigma(t,Y(t),\alpha(t),u_{1}(t),u_{2}(t))dW(t) \nonumber\\
&+\int_{\mathbb{R}_{0}}\eta(t,Y(t-),\alpha(t-),u_{1}(t-),u_{2}(t-),z)\tilde{N}_{\alpha}(dt,dz) \nonumber \\
&+\gamma(t,Y(t-),\alpha(t-),u_{1}(t-),u_{2}(t-))d\tilde{\Phi}(t), \qquad t\in [0,T],  \label{eq:3.1}\\
Y(0)=&\ y_{0}\in \mathbb{R}^{N} \label{3.1.1},
\end{align}
where $\mathcal{U}_{1}$ and $\mathcal{U}_{2}$ are non-empty subsets of $\mathbb{R}^{N}$ and $u_{1}\in \mathcal{U}_{1}$ and $u_{2}\in \mathcal{U}_{2}$ are $\mathcal{F}_{t}$-predictable, c\'adl\'ag (right continuous with left limits) control processes such that 
\begin{equation*}
E\biggl[\int_{0}^{T}\left|u_{k}(t)\right|^{2}dt\biggr]<\infty, \qquad k=1,2.
\end{equation*} 
Furthermore,
\begin{align*}
 &b:[0,T]\times \mathbb{R}^{N}\times S\times \mathcal{U}_{1 }\times \mathcal{U}_{2}\rightarrow \mathbb{R}^{N}, \quad \sigma:[0,T]\times \mathbb{R}^{N}\times S\times \mathcal{U}_{1}\times \mathcal{U}_{2}\rightarrow \mathbb{R}^{N\times M},\\
&\eta:[0,T]\times \mathbb{R}^{N}\times S\times \mathcal{U}_{1}\times \mathcal{U}_{2}\times \mathbb{R}_{0}\rightarrow \mathbb{R}^{N\times L}, \quad
\gamma:[0,T]\times \mathbb{R}^{N}\times S\times \mathcal{U}_{1}\times \mathcal{U}_{2}\rightarrow \mathbb{R}^{N\times D} 
\end{align*}
are given measurable functions with respect to \ $\mathbb{F}$ \ such that
\begin{align*}
&\int_{0}^{T}\biggl\{\left|b(t,Y(t),\alpha(t),u_{1}(t),u_{2}(t))\right|+\left|\sigma(t,Y(t),\alpha(t),u_{1}(t),u_{2}(t))\right|^{2}\\
&\quad+\int_{\mathbb{R}_{0}}\left|\eta(t,Y(t-),\alpha(t-),u_{1}(t-),u_{2}(t-),z)\right|^{2}\nu(dz)
\end{align*}
\begin{align*}
&\quad+\sum_{j=1}^{D}\left|\gamma(t,Y(t-),\alpha(t-),u_{1}(t-),u_{2}(t-))\right|^{2}\mu_{j}(t)\biggr\}dt<\infty.
\end{align*}
Moreover, let $f:[0,T]\times \mathbb{R}^{N}\times S\times \mathcal{U}_{1}\times \mathcal{U}_{2}\rightarrow \mathbb{R}$, called \textit{profit rate}, and \ $g:\mathbb{R}^{N}\times S\rightarrow \mathbb{R}$, called \textit{terminal gain} or \textit{bequest function}, be $C^{1}$ functions with respect to $y$. Then, we can define the performance (objective) functional as follows:
\begin{equation*}
J(y, e_{i},u_{1},u_{2})=E^{y, e_{i}}\biggl[\int_{0}^{T}f(s,Y(s),\alpha(s),u_{1}(s),u_{2}(s))ds+g(Y^{u_{1},u_{2}}(T),\alpha(T))\biggr], 
\end{equation*}
for each $i=1,2,\ldots,D.$
We call the control processes $(u_{1},u_{2})$ are admissible and assume that $\Theta_{1}$ and $\Theta_{2}$ are given families of admissible control processes of $u_{1}\in \mathcal{U}_{1}$ and $u_{2}\in \mathcal{U}_{2}$, respectively, if the following conditions are satisfied: 
\begin{enumerate}
\item There exists a unique strong solution of the state process $Y(t)$ introduced in Equations (\ref{eq:3.1})-(\ref{3.1.1}) (see Proposition 7.1 in \cite{c:56} for an existence-uniqueness theorem of such a system). 
\item $E\biggl[\int_{0}^{T}|f(t,Y(t),\alpha(t),u_{1}(t),u_{2}(t))|dt+|g(Y^{u_{1},u_{2}}(T),\alpha(T))|\biggr]<\infty.$ 
\end{enumerate}


\section{A Zero-Sum Stochastic Differential Game Approach}
\label{sec:3}
Firstly, let us remember the mathematical definition of a saddle point, which are the optimal control processes $(u_{1}^{*},u_{2}^{*})\in  \Theta_{1}\times \Theta_{2}$ \ of a zero-sum stochastic differential game problem (if they exist):  \\
As we described in \cite{savku}, assume that 
\begin{equation*}
J(y, e_{i},u_{1}^{*},u_{2}^{*})\geq J(y, e_{i},u_{1},u_{2}^{*}) \ \hbox{for all \ $u_{1}\in \Theta_{1}$, \  $e_{i}\in S, \ i=1,2,\ldots, D$},
\end{equation*}
where we define:
\begin{equation*}
J(y, e_{i},u_{1}^{*},u_{2}^{*})= \sup_{u_{1}\in \Theta_{1}}J(y, e_{i},u_{1},u_{2}^{*}).
\end{equation*}
Furthermore, suppose that 
\begin{equation*}
J(y, e_{i},u_{1}^{*},u_{2}^{*})\leq J(y, e_{i},u_{1}^{*},u_{2}) \ \hbox{for all \ $u_{2}\in \Theta_{2}$ \ $e_{i}\in S, \ i=1,2,\ldots, D$},
\end{equation*}
where we specify:
\begin{equation*}
J(y, e_{i},u_{1}^{*},u_{2}^{*})= \inf_{u_{2}\in \Theta_{2}}J(y, e_{i},u_{1}^{*},u_{2}).
\end{equation*}
Then, $(u_{1}^{*},u_{2}^{*})$ is a saddle point of a zero-sum stochastic differential game and 
\begin{equation*}
\phi(y,e_{i})=J(y, e_{i},u_{1}^{*},u_{2}^{*})=\ \sup_{u_{1}\in  \Theta_{1}}\biggl(\inf_{u_{2}\in \Theta_{2}}J(y, e_{i},u_{1},u_{2})\biggr)=\ \inf_{u_{2}\in \Theta_{2}}\biggl(\sup_{u_{1}\in  \Theta_{1}}J(y, e_{i},u_{1},u_{2})\biggr)
\end{equation*}
for each $e_{i}\in S,\ i=1,2,\ldots, D$. \\
Now, we can express our constrained and unconstrained zero-sum stochastic differential game formulations and their relations.\\
Our constrained zero-sum problem is to find $(u_{1}^{*},u_{2}^{*})$ for the following system:
\begin{align}
\label{conszero}
&\phi(y,e_{i})=\ \sup_{u_{1}\in  \Theta_{1}}\biggl(\inf_{u_{2}\in \Theta_{2}}J(y, e_{i},u_{1},u_{2})\biggr)=\ \inf_{u_{2}\in \Theta_{2}}\biggl(\sup_{u_{1}\in  \Theta_{1}}J(y, e_{i},u_{1},u_{2})\biggr) \nonumber \\
&=\ \sup_{u_{1}\in  \Theta_{1}}\biggl(\inf_{u_{2}\in \Theta_{2}}E^{y, e_{i}}\biggl[\int_{0}^{T}f(s,Y(s),\alpha(s),u_{1}(s),u_{2}(s))ds+g(Y^{u_{1},u_{2}}(T),\alpha(T))\biggr]\biggr) , \
\end{align}
for $i=1,2,\ldots,D$,\ subject to the system (\ref{eq:3.1})-(\ref{3.1.1}) and the constraints,
\begin{equation}
\label{cons1}
\text{\textbf{(i)}}\ E^{y, e_{i}}[M(Y^{u_{1},u_{2}}(T),\alpha(T))]=0 
\end{equation}
or
\begin{equation}
\label{cons2}
\text{\textbf{(ii)}} \ M(Y^{u_{1},u_{2}}(T),\alpha(T))=0 \quad a.s.,
\end{equation}
 where $M:\mathbb{R}^{N}\rightarrow \mathbb{R}$ is a $C^{1}$ function with respect to $y$.  \\
Here, we introduce two types of constraints. For type (\ref{cons1}), it is enough to determine a real valued Lagrange multiplier, while we have to find out a stochastic one for the stochastic constraint (\ref{cons2}). \\
Hence, we clarify the set of stochastic Lagrange multipliers by:
\begin{equation*}
\Delta=\left\{ \lambda:\Omega\rightarrow \mathbb{R}| \lambda\  \text{is} \ \mathcal{F}_{T}-\text{measurable and} \ E[\lambda]<\infty \right\}.
\end{equation*}
Moreover, in this case, we assume that $E[M(Y^{u_{1},u_{2}}(T),\alpha(T))]< \infty$.\\
Now, we can define our unconstrained zero-sum stochastic differential game as follows:
\begin{align}
\label{unconszero}
\phi^{\lambda}(y,e_{i})&=\ \sup_{u_{1}\in  \Theta_{1}}\biggl(\inf_{u_{2}\in \Theta_{2}}J(y, e_{i},u_{1}^{\lambda},u_{2}^{\lambda})\biggr)=\ \inf_{u_{2}\in \Theta_{2}}\biggl(\sup_{u_{1}\in  \Theta_{1}}J(y, e_{i},u_{1}^{\lambda},u_{2}^{\lambda})\biggr) \nonumber \\
&=\ \sup_{u_{1}\in  \Theta_{1}}\biggl(\inf_{u_{2}\in \Theta_{2}}E^{y, e_{i}}\biggl[\int_{0}^{T}f(t,Y(t),\alpha(t),u_{1}(t),u_{2}(t))dt+g(Y^{u_{1},u_{2}}(T),\alpha(T)) \nonumber\\
&\qquad+\lambda M(Y^{u_{1},u_{2}}(T),\alpha(T))\biggr]\biggr).
\end{align}
for $i=1,2,\ldots,D$,\ subject to the system (\ref{eq:3.1})-(\ref{3.1.1}). \\
Let us provide the following theorem for the constraint type (\ref{cons2}):
\begin{theorem}
\label{th1}
Suppose that for all $\lambda \in \Delta_{1}\subset \Delta$, we can find $\phi^{\lambda}(y,e_{i})$, $i=1,2,\ldots, D$, and a saddle point $(u_{1}^{*,\lambda},u_{2}^{*,\lambda})$ solving the unconstrained stochastic control problem (\ref{unconszero}) subject to (\ref{eq:3.1})-(\ref{3.1.1}). Moreover, suppose that there exists $\lambda_{0}\in \Delta_{1}$, such that 
\begin{equation}
\label{th1cons2}
M(Y_{T}^{u_{1}^{*,\lambda_{0}},u_{2}^{*,\lambda_{0}}},e_{i})= 0, \qquad a.s.
\end{equation}
for all $e_{i}\in S, \ i=1,2,\ldots, D$.\\
Then, $\phi(y,e_{i})=\phi^{\lambda_{0}}(y,e_{i}), \ i=1,2,\ldots,D$ and $(u_{1}^{*},u_{2}^{*})=(u_{1}^{*,\lambda_{0}},u_{2}^{*,\lambda_{0}})$ solves the constrained stochastic control problem (\ref{conszero}) subject to (\ref{eq:3.1})-(\ref{3.1.1}) and (\ref{cons2}).
\end{theorem}
\begin{proof}
By definition of the saddle-point, we have 
\begin{align}
\phi^{\lambda}(y,e_{i})&=\ J(y,e_{i},u_{1}^{*,\lambda},u_{2}^{*,\lambda})= E^{y, e_{i}}\biggl[\int_{0}^{T}f(t,Y_{t}^{u_{1}^{*,\lambda},u_{2}^{*,\lambda}},e_{i},u_{1}^{*,\lambda},u_{2}^{*,\lambda})dt+ g(Y_{T}^{u_{1}^{*,\lambda},u_{2}^{*,\lambda}},\alpha_{T}) \nonumber\\
&\qquad \qquad +\lambda M(Y_{T}^{u_{1}^{*,\lambda},u_{2}^{*,\lambda}},\alpha_{T})\biggr] \nonumber \\
&\geq J(y,e_{i},u_{1}^{\lambda},u_{2}^{*,\lambda})=E^{y, e_{i}}\biggl[\int_{0}^{T}f(t,Y_{t}^{u_{1}^{\lambda},u_{2}^{*,\lambda}},e_{i},u_{1}^{\lambda},u_{2}^{*,\lambda})dt+ g(Y_{T}^{u_{1}^{\lambda},u_{2}^{*,\lambda}},\alpha_{T}) \nonumber \\
&\qquad \qquad +\lambda M(Y_{T}^{u_{1}^{\lambda},u_{2}^{*,\lambda}},\alpha_{T})\biggr].\label{p2}
\end{align} 
For the optimal strategy of Player 2, $u_{2}^{*,\lambda}\in \Theta_{2}, \ \lambda\in \Delta_{1}$, in particular if $\lambda=\lambda_{0}$ and since $u_{1}\in \Theta_{1}$ is feasible in the constrained control problem, then by (\ref{th1cons2})
\begin{equation*}
M(Y_{T}^{u_{1}^{*,\lambda_{0}},u_{2}^{*,\lambda_{0}}},e_{i}) = 0 =\ M(Y_{T}^{u_{1},u_{2}^{*}},e_{i}),\qquad \text{for} \quad i=1,2,\ldots,D . \\
\end{equation*}
By (\ref{p2}),
\begin{equation}
\label{p21}
\phi^{\lambda_{0}}(y,e_{i})=\ J(y,e_{i},u_{1}^{*,\lambda_{0}},u_{2}^{*,\lambda_{0}})\geq J(y,e_{i},u_{1},u_{2}^{*}),
\end{equation}
for all $u_{1}\in \Theta_{1}$ and $e_{i}\in S, \ i=1,2,\ldots,D$.\\
Moreover, we know that  
\begin{align}
\phi^{\lambda}(y,e_{i})&=\ J(y,e_{i},u_{1}^{*,\lambda},u_{2}^{*,\lambda})= E^{y, e_{i}}\biggl[\int_{0}^{T}f(t,Y_{t}^{u_{1}^{*,\lambda},u_{2}^{*,\lambda}},e_{i},u_{1}^{*,\lambda},u_{2}^{*,\lambda})dt+ g(Y_{T}^{u_{1}^{*,\lambda},u_{2}^{*,\lambda}},\alpha_{T}) \nonumber\\
&\qquad \qquad +\lambda M(Y_{T}^{u_{1}^{*,\lambda},u_{2}^{*,\lambda}},\alpha_{T})\biggr] \nonumber \\
&\leq J(y,e_{i},u_{1}^{*,\lambda},u_{2}^{\lambda})=E^{y, e_{i}}\biggl[\int_{0}^{T}f(t,Y_{t}^{u_{1}^{*,\lambda},u_{2}^{\lambda}},e_{i},u_{1}^{*,\lambda},u_{2}^{\lambda})dt+ g(Y_{T}^{u_{1}^{*,\lambda},u_{2}^{\lambda}},\alpha_{T}) \nonumber \\
&\qquad \qquad +\lambda M(Y_{T}^{u_{1}^{*,\lambda},u_{2}^{\lambda}},\alpha_{T})\biggr].\label{p3}
\end{align} 
for all $u_{2}\in \Theta_{2}$ and $e_{i}\in S, \ i=1,2,\ldots,D$.\\
For the optimal strategy of Player 1, $u_{1}^{*,\lambda}\in \Theta_{1}, \ \lambda\in \Delta_{1}$, in particular if $\lambda=\lambda_{0}$ and since $u_{2}\in \Theta_{2}$ is feasible in the constrained control problem, then by (\ref{th1cons2})
\begin{equation*}
M(Y_{T}^{u_{1}^{*,\lambda_{0}},u_{2}^{*,\lambda_{0}}},e_{i})= 0 =M(Y_{T}^{u_{1}^{*},u_{2}},e_{i}), \qquad a.s. \ \text{for} \quad i=1,2,\ldots,D. \\
\end{equation*}
By (\ref{p3}),
\begin{equation}
\label{p32}
\phi^{\lambda_{0}}(y,e_{i})=\ J(y,e_{i},u_{1}^{*,\lambda_{0}},u_{2}^{*,\lambda_{0}})\leq J(y,e_{i},u_{1}^{*},u_{2}),
\end{equation}
for all $u_{2}\in \Theta_{2}$ and $e_{i}\in S, \ i=1,2,\ldots,D$.\\
Consequently, we obtain by (\ref{p21})-(\ref{p32})
\begin{equation*}
J(y,e_{i},u_{1},u_{2}^{*})\leq J(y,e_{i},u_{1}^{*,\lambda_{0}},u_{2}^{*,\lambda_{0}})=\ \phi^{\lambda_{0}}(y,e_{i})\leq J(y,e_{i},u_{1}^{*},u_{2})
\end{equation*}
for any feasible $(u_{1},u_{2})\in \Theta_{1}\times \Theta_{2}$ \ and for all \  $e_{i}\in S, \ i=1,2,\ldots,D$. \\
Then,
\begin{equation*}
J(y,e_{i},u_{1}^{*,\lambda_{0}},u_{2}^{*,\lambda_{0}})\leq \inf_{u_{2}\in \Theta_{2}} J(y,e_{i},u_{1}^{*},u_{2})\leq \sup_{u_{1}\in \Theta_{1}}\biggl(\inf_{u_{2}\in \Theta_{2}}J(y, e_{i},u_{1},u_{2})\biggr).
\end{equation*}
Moreover,
\begin{equation*}
J(y,e_{i},u_{1}^{*,\lambda_{0}},u_{2}^{*,\lambda_{0}})\geq \sup_{u_{1}\in \Theta_{1}}J(y,e_{i},u_{1},u_{2}^{*})\geq \inf_{u_{2}\in \Theta_{2}}\biggl(\sup_{u_{1}\in \Theta_{1}}J(y,e_{i},u_{1},u_{2})\biggr)
\end{equation*}
Hence, we obtain:
\begin{equation*}
\sup_{u_{1}\in \Theta_{1}}\biggl(\inf_{u_{2}\in \Theta_{2}}J(y,e_{i},u_{1},u_{2})\biggr)\geq \inf_{u_{2}\in \Theta_{2}}\biggl(\sup_{u_{1}\in \Theta_{1}}J(y, e_{i},u_{1},u_{2})\biggr).
\end{equation*}
Since we always have 
\begin{equation*}
\sup_{u_{1}\in \Theta_{1}}\biggl(\inf_{u_{2}\in \Theta_{2}}J(y,e_{i},u_{1},u_{2})\biggr)\leq \inf_{u_{2}\in \Theta_{2}}\biggl(\sup_{u_{1}\in \Theta_{1}}J(y, e_{i},u_{1},u_{2})\biggr),
\end{equation*}
finally, we prove
\begin{equation*}
\phi(y, e_{i})=\sup_{u_{1}\in  \Theta_{1}}\biggl(\inf_{u_{2}\in \Theta_{2}}J(y, e_{i},u_{1},u_{2})\biggr)=\inf_{u_{2}\in \Theta_{2}}\biggl(\sup_{u_{1}\in  \Theta_{1}}J(y, e_{i},u_{1},u_{2})\biggr)=\phi^{\lambda_{0}}(y,e_{i}), 
\end{equation*}
for \ $i=1,2,\ldots,D$.\\
This completes the proof.
\end{proof}
We can prove the following theorem similarly for the constraint type (\ref{cons1}).
\begin{theorem}
\label{th2}
Suppose that for all $\lambda \in K\subset \mathbb{R}$, we can find $\phi^{\lambda}(y,e_{i})$, $i=1,2,\ldots, D$, and a saddle point $(u_{1}^{*,\lambda},u_{2}^{*,\lambda})$ solving the unconstrained stochastic control problem (\ref{unconszero}) subject to (\ref{eq:3.1})-(\ref{3.1.1}). Moreover, suppose that there exists $\lambda_{0}\in K$, such that 
\begin{equation*}
\label{th2cons1}
E[M(Y_{T}^{u_{1}^{*,\lambda_{0}},u_{2}^{*,\lambda_{0}}},e_{i})]= 0, 
\end{equation*}
for all $e_{i}\in S, \ i=1,2,\ldots, D$.\\
Then, $\phi(y,e_{i})=\phi^{\lambda_{0}}(y,e_{i}), \ i=1,2,\ldots,D$ and $(u_{1}^{*},u_{2}^{*})=(u_{1}^{*,\lambda_{0}},u_{2}^{*,\lambda_{0}})$ solves the constrained stochastic control problem (\ref{conszero}) subject to (\ref{eq:3.1})-(\ref{3.1.1}) and (\ref{cons1}).
\end{theorem}
In this section, we extended Theorem 11.3.1 in \cite{oksendal2} to a zero-sum stochastic differential game formulation within the framework of regime-switches.


\section{A Nonzero-Sum Stochastic Differential Game Approach}
\label{sec:4}

By solving a nonzero-sum stochastic differential game, we purpose to find a pair of optimal control processes, which corresponds to \textit{Nash equilibrium} of the two player game, if it exists. Remember that Nash equilibrium is a self-enforcing strategy; i.e., each player knows that unilateral profitable deviation is not possible. This also means that each player's strategy is optimal or a best response against the other one's.   \\
Then, a mathematical definition of the Nash equilibrium can be introduced as we described in \cite{savku}:\\
Let $u_{1}\in \Theta_{1}$ and $u_{2}\in \Theta_{2}$ be two admissible control processes for Player 1 and Player 2, respectively. We define the performance criterion for each player as follows:
\begin{equation*}
J_{k}(y, e_{i},u_{1},u_{2})=E^{y, e_{i}}\biggl[\int_{0}^{T}f_{k}(s,Y(s),\alpha(s),u_{1}(s),u_{2}(s))ds+g_{k}(Y^{u_{1},u_{2}}(T),\alpha(T))\biggr]
\end{equation*}  
for each $e_{i}\in S$, \ $i=1,2,\ldots,D$, and both purpose to maximize their payoffs with respect to other player's best action as follows:
\begin{align}
&J_{1}(y, e_{i},u_{1}^{*},u_{2}^{*})=\ \sup_{u_{1}\in  \Theta_{1}}J_{1}(y, e_{i},u_{1},u_{2}^{*}), \label{j1}\\
&J_{2}(y, e_{i},u_{1}^{*},u_{2}^{*})=\ \sup_{u_{2}\in  \Theta_{2}}J_{2}(y, e_{i},u_{1}^{*},u_{2}), \label{j2}
\end{align}
for each $e_{i}\in S$ and for all $y\in G$, where $G$ is an open subset of $\mathbb{R}^{N}$ and corresponds to a solvency region for the state processes.
\begin{definition} 
(Definition 1, \cite{savku}) Let us assume that for the optimal strategy of Player 2, $u_{2}^{*}\in \Theta_{2}$, the best response of Player 1 satisfies
\begin{equation*}
J_{1}(y, e_{i},u_{1},u_{2}^{*})\leq J_{1}(y, e_{i},u_{1}^{*},u_{2}^{*}) \ \hbox{for all} \ u_{1}\in \Theta_{1}, \ e_{i}\in S, \ y\in G,
\end{equation*}
and for the optimal strategy of Player 1, $u_{1}^{*}\in \Theta_{1}$, the best response of Player 2 satisfies
\begin{equation*}
J_{2}(y, e_{i},u_{1}^{*},u_{2})\leq J_{2}(y,e_{i},u_{1}^{*},u_{2}^{*}) \ \hbox{for all} \ u_{2}\in \Theta_{2}, \ e_{i}\in S, \ y\in G.
\end{equation*}
Then, the pair of optimal control processes $(u_{1}^{*},u_{2}^{*})\in \Theta_{1}\times \Theta_{2}$ is called a Nash equilibrium for the stochastic differential game of the system (\ref{eq:3.1})-(\ref{3.1.1}) and (\ref{j1})-(\ref{j2}).
\end{definition}
Our constrained nonzero-sum stochastic differential game is to find out $(u_{1}^{*},u_{2}^{*})$ for the problems (\ref{j1})-(\ref{j2}) subject to the system (\ref{eq:3.1})-(\ref{3.1.1}) and 
 \begin{equation}
\label{cons11}
\text{\textbf{(i)}}\ E^{y, e_{i}}[M_{k}(Y^{u_{1},u_{2}}(T),\alpha(T))]=0 
\end{equation}
or
\begin{equation}
\label{cons22}
\text{\textbf{(ii)}} \ M_{k}(Y^{u_{1},u_{2}}(T),\alpha(T))=0 \quad a.s.,
\end{equation}
 where $M_{k}:\mathbb{R}^{N}\rightarrow \mathbb{R}$, \  are \ $C^{1}$ \ functions with respect to\ $y$ \ and we assume that \ $E[M_{k}(Y^{u_{1},u_{2}}(T),\alpha(T))]< \infty$, \ $k=1,2$.\\

Finally, our unconstrained nonzero-sum stochastic differential game problem is described as follows:
\begin{align}
\phi_{k}^{\lambda_{k}}(y,e_{i})=&\ J_{k}(y, e_{i},u_{1}^{*,\lambda_{1}},u_{2}^{*,\lambda_{2}})=\ \sup_{u_{k}\in \Theta_{k}}E^{y, e_{i}}\biggl[\int_{0}^{T}f_{k}(t,Y(t),\alpha(t),u_{1}(t),u_{2}(t))dt \nonumber \\
&\qquad +g_{k}(Y^{u_{1},u_{2}}(T),\alpha(T))+\lambda_{k}M_{k}(Y^{u_{1},u_{2}}(T),\alpha(T))\biggr] \label{unconsnonzero} 
\end{align}
for $k=1,2$ and $e_{i}\in S, \ i=1,2,\ldots, D$, subject to the system (\ref{eq:3.1})-(\ref{3.1.1}).
\begin{theorem} 
\label{th3}
Suppose that for all $\lambda_{k} \in \Delta_{k} \subseteq \Delta$, we can find $\phi_{k}^{\lambda_{k}}(y,e_{i})$, \ $i=1,2,\ldots, D$ \ and \ $k=1,2$, and a Nash equilibrium $(u_{1}^{*,\lambda_{1}},u_{2}^{*,\lambda_{2}})$ solving the unconstrained stochastic control problems (\ref{unconsnonzero}) for each player. Moreover, suppose that there exist $\lambda_{k}^{0}\in \Delta_{k}\subseteq \Delta$, \ $k=1,2$, such that 
\begin{equation}
\label{th3cons2}
M_{1}(Y_{T}^{u_{1}^{*,\lambda_{1}^{0}},u_{2}^{*,\lambda_{2}}},e_{i})= 0 \quad \text{and}\quad  M_{2}(Y_{T}^{u_{1}^{*,\lambda_{1}},u_{2}^{*,\lambda_{2}^{0}}},e_{i})= 0 \quad a.s.,
\end{equation}
for all $e_{i}\in S, \ i=1,2,\ldots, D$.\\
Then, $\phi_{k}(y,e_{i})=\phi_{k}^{\lambda_{k}^{0}}(y,e_{i}), \ k=1,2, \ i=1,2, \ldots, D$ and $(u_{1}^{*},u_{2}^{*})=(u_{1}^{*,\lambda_{1}^{0}},u_{2}^{*,\lambda_{2}^{0}})$ solves the constrained stochastic control problem.
\end{theorem}
\begin{proof}
By definition of Nash equilibrium, we have 
\begin{align}
&J_{1}(y,e_{i},u_{1}^{*,\lambda_{1}},u_{2}^{*,\lambda_{2}})= E^{y, e_{i}}\biggl[\int_{0}^{T}f_{1}(t,Y_{t}^{u_{1}^{*,\lambda_{1}},u_{2}^{*,\lambda_{2}}},e_{i},u_{1}^{*,\lambda_{1}},u_{2}^{*,\lambda_{2}})dt+ g_{1}(Y_{T}^{u_{1}^{*,\lambda_{1}},u_{2}^{*,\lambda_{2}}},\alpha_{T}) \nonumber\\
&\qquad \qquad +\lambda_{1} M_{1}(Y_{T}^{u_{1}^{*,\lambda_{1}},u_{2}^{*,\lambda_{2}}},\alpha_{T})\biggr] \nonumber \\
&\geq J_{1}(y,e_{i},u_{1}^{\lambda_{1}},u_{2}^{*,\lambda_{2}})=E^{y, e_{i}}\biggl[\int_{0}^{T}f_{1}(t,Y_{t}^{u_{1}^{\lambda_{1}},u_{2}^{*,\lambda_{2}}},e_{i},u_{1}^{\lambda_{1}},u_{2}^{*,\lambda_{2}})dt+ g_{1}(Y_{T}^{u_{1}^{\lambda_{1}},u_{2}^{*,\lambda_{2}}},\alpha_{T}) \nonumber \\
&\qquad \qquad +\lambda_{1} M_{1}(Y_{T}^{u_{1}^{\lambda_{1}},u_{2}^{*,\lambda_{2}}},\alpha_{T})\biggr].\label{p1}
\end{align} 
For the optimal strategy of Player 2, $u_{2}^{*,\lambda_{2}}\in \Theta_{2}, \ \lambda_{2}\in \Delta_{2}$, in particular if $\lambda_{1}=\lambda_{1}^{0}$ and since $u_{1}\in \Theta_{1}$ is feasible in the constrained control problem, then by (\ref{cons22}) and (\ref{th3cons2}).
\begin{equation*}
M_{1}(Y_{T}^{u_{1}^{*,\lambda_{1}^{0}},u_{2}^{*,\lambda_{2}}},e_{i})=\ 0 =\ M_{1}(Y_{T}^{u_{1},u_{2}^{*}},e_{i}), \quad a.s. \quad \text{for} \quad i=1,2,\ldots,D.\\
\end{equation*}
By (\ref{p1}),
\begin{equation}
\label{p11}
J_{1}(y,e_{i},u_{1}^{*,\lambda_{1}^{0}},u_{2}^{*,\lambda_{2}})\geq J_{1}(y,e_{i},u_{1},u_{2}^{*}),
\end{equation}
for all $u_{1}\in \Theta_{1}$ and $e_{i}\in S, \ i=1,2,\ldots,D$.\\
Similarly, we can obtain 
\begin{equation}
\label{p12}
J_{2}(y,e_{i},u_{1}^{*,\lambda_{1}},u_{2}^{*,\lambda_{2}^{0}})\geq J_{2}(y,e_{i},u_{1}^{*},u_{2}),
\end{equation}
for all $u_{2}\in \Theta_{2}$ and $e_{i}\in S, \ i=1,2,\ldots,D$.\\
Hence, by the definition of Nash equilibrium, the inequalities (\ref{p11})-(\ref{p12}) complete the proof.
\end{proof}
We can easily develop a similar theorem for the constraint type (\ref{cons11}) as well:
\begin{theorem} 
\label{th4}
Suppose that for all $\lambda_{k} \in A_{k} \subseteq \mathbb{R}$, we can find $\phi_{k}^{\lambda_{k}}(y,e_{i})$, \ $i=1,2,\ldots, D$ \ and \ $k=1,2$, and a Nash equilibrium $(u_{1}^{*,\lambda_{1}},u_{2}^{*,\lambda_{2}})$ solving the unconstrained stochastic control problems (\ref{unconsnonzero}) for each player. Moreover, suppose that there exist $\lambda_{k}^{0}\in A_{k}\subseteq \mathbb{R}$, \ $k=1,2$, such that 
\begin{equation}
\label{th4cons1}
E[M_{1}(Y_{T}^{u_{1}^{*,\lambda_{1}^{0}},u_{2}^{*,\lambda_{2}}},e_{i})]= 0  \quad \text{and}\quad  E[M_{2}(Y_{T}^{u_{1}^{*,\lambda_{1}},u_{2}^{*,\lambda_{2}^{0}}},e_{i})]= 0,
\end{equation}
for all $e_{i}\in S, \ i=1,2,\ldots, D$.\\
Then, $\phi_{k}(y,e_{i})=\phi_{k}^{\lambda_{k}^{0}}(y,e_{i}), \ k=1,2, \ i=1,2, \ldots, D$ and $(u_{1}^{*},u_{2}^{*})=(u_{1}^{*,\lambda_{1}^{0}},u_{2}^{*,\lambda_{2}^{0}})$ solves the constrained stochastic control problem described above.
\end{theorem}
In this section, we extended Theorem 11.3.1 in \cite{oksendal2} to a nonzero-sum stochastic differential game formulation within the framework of regime-switches.
\begin{remark}
Firstly, we should indicate that Theorems \ref{th1}, \ref{th2}, \ref{th3} and \ref{th4} can be applied to both of the dynamic programing principle and stochastic maximum principle under specific technical conditions, which arise as a consequence of the nature of the corresponding technique. \\
In this work, we provide an application of nonzero-sum game formulation with stochastic maximum principle. The technical conditions and problem formulation within the framework of stochastic maximum principle for a Markov regime-switching jump diffusion system of such a game have already been developed in \cite{pamen} without a Lagrangian setting. Hence, here, we just remind an appropriate version of sufficient maximum principle in Appendix.  \\
On the other side, while applying Theorem \ref{suff} to a Lagrangian problem, one should be careful about that $g_{k}^{\lambda_{k}}(y,e_{i})=g_{k}(y,e_{i})+\lambda_{k}(M_{k}(y,e_{i}))$'s are concave, $C^{1}$-functions with respect to $y$ for $i=1,2,\ldots,D$ \ and \ $k=1,2$.
\end{remark}


\section{An Application: Bancassurance}
\label{apply}

In this section, we provide an application of Theorem \ref{th4} within the framework of a collaboration between a bank and an insurance company, which illustrates an example of a well-known concept: \textit{Bancassurance}. \\
The core of the joint venture of the bank and insurers is to strengthen their business objectives through sharing a client database, development of the products as well as the coordination. Insurance companies and banks utilize from this long-term cooperation in many senses. For example, insurance companies may create new and more efficient financial instruments with the help of the experience of the banks. Furthermore, they can reach the wide customer portfolio of the banks without investing in more offices and manpower. Hence, the insurance companies may reduce the costs while they increase the sales. On the other side, banks can also increase their income, diversify offered financial products, and by providing different services under one roof, they can gain more customer loyalty and satisfaction. Some financial and actuarial aspects for this legal and independent organizational entity, Bancassurance, can be found in \cite{banc1}, \cite{banc2}, \cite{banc3} and references therein.   \\
Basically, in our formulation, insurance company gives a certain amount of its surplus to the bank as a commission, which becomes the initial value of the wealth process of the bank. Moreover, in this application we can approach to our regime-switching model from two different sides:
\begin{enumerate}
\item The states of the Markov chain may represent switches in different states of the economy such as a shift from recession to growth periods, macroeconomic indicators, regulation changes or a radical change like a financial crisis. All these impulses affect the cash flows of both of the insurance company and the bank within the framework of decision making, dividend management, commissions, number of the clients etc. In this case, we may also assume to use a time-homogenous, irreducible Markov chain.\\
\item We can see the shifts as the states of a life insurance policy; i.e. the states of the Markov chain may represent injured, dead and alive cases. We may assume that the bank make an investment, whose cash flow is impacted by the abrupt changes experienced by the insured, like investing in the stocks of an insurance company. In this case, time-inhomogenous and reducible Markov chains may be utilized.
\end{enumerate}
First, let us introduce the dynamics of the wealth process of the insurance company. $\tilde{p}(t)$ represents deterministic premium rate at time $t\in [0,T]$ for each claim. An insurance premium is the amount of money an insured has to pay for an insurance policy, which may cover healthcare, auto, home, and life insurance.
The premium is an income for an insurance company, which is used for coverage of the claims against the policy. Additionally, the companies may utilize premiums to make some investments and increase their own wealth. In our application, we focus on a life insurance policy.\\
Now, we can show the dynamics of the surplus process by $R(t)$ as follows:
\begin{align*}
dR(t)=&\ a(t, \alpha(t-))dt+\sigma_{1}(t,\alpha(t-))dW_{1}(t)+\sum_{i=1,\ i\neq j}^{N}\gamma^{ij}(t)(dJ^{ij}(t)-\mu_{ij}(t))\\
=&\ a(t, \alpha(t-))dt+\sigma_{1}(t,\alpha(t-))dW_{1}(t)+\gamma(t,\alpha(t-))d\tilde{\Phi}(t)
\end{align*}
Here, for $t\in [0,T]$, \ $\sigma_{1}(t,\alpha(t-))$ denotes the instantaneous volatility of the aggregate insurance claims at time $t$ \ and \ $a(t, \alpha(t-))$  specifies the payments of the insurer due during sojournus in state $i$. Finally, $\gamma^{ij}(\cdot)$ determines the claims of the insurance company to the insured due upon transition from state $i$ to $j$.\\
In this context, $\alpha(\cdot)$ is a time-inhomogenous Markov chain and $\mu_{ij}(\cdot)$ indicates the intensity of the chain, which corresponds to the mortality rate for a life-insurance contract. \\
Furthermore, $J^{ij}(t)$ denotes the number of the transitions into state $j$ up to and including time $t\in [0,T]$ for the associated c\'adl\'ag counting process, for which we use a martingale form. By the way, we describe a risk exchange between insurer and insured. While the company pays out the amount of insurance claim $\gamma^{ij}(t)$ upon a transition to state $j$, policy holder has to pay an amount of $\gamma^{ij}(t)\mu_{ij}(t)$ if she is in state $i$ and $t\in [0,T]$.\\
In this application, we suppose that the insurance company pays an amount of its surplus to its shareholders, which is called dividend distribution, and it is the only control process of the pension fund defined as follows: 
\begin{equation*}
d\tilde{D}(t)=\delta(t)dt.
\end{equation*}
Hence, we represent the wealth (cash) process of the insurance company $X_{1}(t), \ t\in [0,T]$ as follows:
\begin{align}
\label{comp}
dX_{1}(t)=& \tilde{p}(t)dt-dR(t)-dD(t) \nonumber\\
X_{1}(0)=&u-c,
\end{align}
where $u$ and $c$ are nonnegative constant values corresponds to initial surplus of the insurance company and the commission of the bank paid at time $t=0$, respectively.\\
We focus on the cash flow of the bank, which is generated just by an investment of the gathered commissions via the bancassurance agreement rather than other investments of the bank.\\
Let us introduce the wealth (cash) process the bank:
\begin{align*}
dX_{2}(t)=&\ X_{2}(t-)\biggl\{u(t)dt+\sigma_{2}(t,\alpha(t-))dW_{2}(t)+\int_{\mathbb{R}_{0}}\eta(t,\alpha(t-),z)\tilde{N}(dt,dz)\biggr\}\\
X_{2}(0)=&\ c, \quad t\in [0,T]
\end{align*} 
where the appreciation rate is not a given priori. Specifically, $u(\cdot)$ is a control process depending on the interaction between the bank and the insurer.\\
Now, we can present the state process of the system as follows:
\begin{align}
&dY(t)= \nonumber
\begin{bmatrix}
dX_{1}(t)  \nonumber\\
dX_{2}(t)  \nonumber
\end{bmatrix}
\\ & \qquad =\nonumber
\begin{bmatrix}
\tilde{p}(t)-a(t,\alpha(t-))-\delta(t) \nonumber\\
X_{2}(t-)u(t) \nonumber
\end{bmatrix}
dt+\nonumber
\begin{bmatrix}
-\sigma_{1}(t,\alpha(t-)) & 0   \nonumber   \\
0 &  X_{2}(t-)\sigma_{2}(t,\alpha(t-))  \nonumber
\end{bmatrix}
\begin{bmatrix}
dW_{1}(t) \nonumber \\
dW_{2}(t)   \nonumber
\end{bmatrix}
\\& \qquad+    
\begin{bmatrix}
0      \\
X_{2}(t-)\int_{\mathbb{R}_{0}}\eta(t,\alpha(t-),z)    
\end{bmatrix}
\tilde{N}(dt,dz)+    
\begin{bmatrix}
-\gamma(t,\alpha(t-)) \\
0                   
\end{bmatrix}
d\tilde{\Phi}(t),     \label{system} \\
&\text{with initial values,}       \nonumber    \\
&Y(0)=                      \nonumber
\begin{bmatrix}
X_{1}(0)           \nonumber \\
X_{2}(0)\nonumber
\end{bmatrix}
=        \nonumber
\begin{bmatrix}
u-c      \nonumber      \\
c        \nonumber
\end{bmatrix}
>0. \nonumber
\end{align}
We assume that $W_{1}$ and $W_{2}$ are independent Brownian motions; moreover, $a, \ \sigma_{1}, \ \sigma_{2},\ \eta$, \ and \ $\gamma$ are square integrable and measurable functions. \\ 
Let us describe the performance functionals of the insurer and the bank by $J_{1}(\delta, u^{*})$ and $J_{2}(\delta^{*}, u)$ correspondingly:
\begin{equation*}
J_{1}(\delta, u^{*})=\ E^{x,e_{i}}\biggl[\int_{0}^{T}\frac{1}{1-\kappa_{1}}h_{1}(t,\alpha(t-))\delta(t)^{1-\kappa_{1}}dt-X^{2}_{2}(T)\biggr],
\end{equation*}                               
and
\begin{equation*}
J_{2}(\delta^{*},u)=\ E^{x,e_{i}}\biggl[\int_{0}^{T}h_{2}(t,\alpha(t-))\ln(u(t))dt+\kappa_{2}X_{1}(T)\biggr],
\end{equation*}
where $\kappa_{1}\geq 0, \ \kappa_{1}\neq 1$, \ $\kappa_{2}\in \mathbb{R}$, \ $h_{1}, \ h_{2}$ \ are square integrable, measurable functions and $\tilde{r}(t,e_{i})=(\tilde{r}_{1},\tilde{r}_{2},\ldots,\tilde{r}_{D})$, \ $i=1,2,\ldots,D$, are constants at each state on $[0,T]$ and can be seen as interest rates in different states of the economy.\\ 
Hence, our problem is to find $(\delta^{*},u^{*})$ by solving:
\begin{equation*}
J_{1}(\delta^{*}, u^{*})=\ \sup_{\delta\in \Theta_{1}}J_{1}(\delta, u^{*}) \\
\end{equation*}
subject to the system (\ref{system}) and
\begin{equation*}
E[X_{1}(T)]=K_{1},
\end{equation*}
and 
\begin{equation*}
J_{2}(\delta^{*}, u^{*})=\ \sup_{u\in \Theta_{2}}J_{1}(\delta^{*}, u) \\
\end{equation*}
subject to the system (\ref{system}) and
\begin{equation*}
E[e^{-\tilde{r}(T,\alpha(T))}\ln(X_{2}(T))]=K_{2}. 
\end{equation*}
Here, in this context, Player 1 wants to maximize $h_{1}(\cdot, e_{i})$ times power utility of her dividend processes while she punishes the deviation of terminal value of Player 2 from $0$ and sets a goal to reach a level of $K_{1}$ for the terminal value of her wealth process in the sense of expected values. On the other side, Player 2 purposes to maximize $h_{2}(\cdot, e_{i})$ times logarithmic utility of her appreciation rate with $\kappa_{1}$ times the terminal value of insurer's wealth process while she sets a target to catch $K_{2}$ level for the discounted logarithm of her terminal value in the sense of expected values.\\
Finally, if we consider this nonzero-sum game problem in terms of the Lagrangian formulation by Theorem \ref{th4}, our problem becomes to find $(\delta^{*},u^{*})$ for:
\begin{equation*}
J_{1}(\delta^{*}, u^{*})=\sup_{\delta\in \Theta_{1}} E^{x,e_{i}}\biggl[\int_{0}^{T}\frac{1}{1-\kappa_{1}}h_{1}(t,\alpha(t-))\delta(t)^{1-\kappa_{1}}dt-X^{2}_{2}(T)+\lambda_{1}(X_{1}(T)-K_{1})\biggr],
\end{equation*}
and
\begin{equation*}
J_{2}(\delta^{*},u^{*})=\sup_{u\in \Theta_{2}} E^{x,e_{i}}\biggl[\int_{0}^{T}h_{2}(t,\alpha(t-))\ln u(t)dt+\kappa_{2}X_{1}(T)+\lambda_{2}(e^{-\tilde{r}(T,\alpha(T))}\ln X_{2}(T)-K_{2})\biggr].
\end{equation*}
Now, we can provide the corresponding Hamiltonian functions for each player and solve them by Theorem \ref{suff} (see Apendix):
\begin{align*}
&H^{1}(t,y,\delta,u, p^{2},q^{2},r^{2}(\cdot),w^{2},e_{i})=\ \frac{1}{1-\kappa_{1}}\delta^{1-\kappa_{1}}h_{1}(t,\alpha(t-))+(\tilde{p}-a(t,e_{i}))-\delta)p_{1}^{1}+x_{2}up^{1}_{2} \nonumber\\
&\quad-\sigma_{1}(t,e_{i})q^{1}_{11}+x_{2}\sigma_{2}(t,e_{i})q^{1}_{22}+x_{2}\int_{\mathbb{R}_{0}}\eta(t-,e_{i},z)r^{1}_{2}(t-,z)\nu(dz)-\sum_{j=1}^{N}\gamma^{ij}(t)w_{1}^{1,j}\mu_{ij}(t).
\end{align*}
and 
\begin{align*}
&H^{2}(t,y,\delta,u, p^{2},q^{2},r^{2}(\cdot),w^{2},e_{i})=\ h_{2}(t,e_{i})\ln(u)+(\tilde{p}-a(t,e_{i}))-\delta)p_{1}^{2}+x_{2}up^{2}_{2} \nonumber\\
&\quad-\sigma_{1}(t,e_{i})q^{2}_{11}+x_{2}\sigma_{2}(t,e_{i})q^{2}_{22}+x_{2}\int_{\mathbb{R}_{0}}\eta(t-,e_{i},z)r^{2}_{2}(t-,z)\nu(dz)+\sum_{j=1}^{D}-\gamma^{ij}(t)w_{1}^{2,j}\mu_{ij}(t).
\end{align*}
We purpose to solve the corresponding BSDEs with jumps and regimes and find the followings:
\begin{align*}
&p^{k}(t)=
\begin{bmatrix}
p^{k}_{1}(t)\\
p^{k}_{2}(t)
\end{bmatrix}
\quad q^{k}(t)=
\begin{bmatrix}
q^{k}_{11}(t) & q^{k}_{12}(t)\\
q^{k}_{21}(t) & q^{k}_{22}(t)
\end{bmatrix}
\quad r^{k}(t-,z)=
\begin{bmatrix}
r^{k}_{1}(t-,z)\\
r^{k}_{2}(t-,z)
\end{bmatrix}
\\
\\ &w(t)=
\begin{bmatrix}
w^{k}_{1}(t)\\
w^{k}_{2}(t)
\end{bmatrix}
=
\begin{bmatrix}
w^{k,j}_{1}(t)\\
w^{k,j}_{2}(t)
\end{bmatrix}
\qquad t\in [0,T], \quad k=1,2 \quad \text{and} \quad j=1,2,\ldots,D.
\end{align*}
Firstly, let us solve the adjoint equations corresponding to $H_{1}$:
\begin{align}
\label{p-1-1}
dp_{1}^{1}(t)=&\ q^{1}_{11}(t)dW_{1}(t)+q^{1}_{12}(t)dW_{2}(t)+\int_{\mathbb{R}_{0}}r^{1}_{1}(t-,z)\tilde{N}(dt,dz)+w_{1}^{1}(t)d\tilde{\Phi}(t),\nonumber\\
p_{1}^{1}(T)=&\ \lambda_{1},
\end{align}
and
\begin{align}
\label{p-1-2}
dp_{2}^{1}(t)=&\ -\biggl(u(t)p^{1}_{2}(t)+\sigma_{2}(t,\alpha(t-))q^{1}_{22}(t)+\int_{\mathbb{R}_{0}}\eta(t-,e_{i},z)r^{1}_{2}(t-,z)\nu(dz)\biggr)dt \nonumber\\
&+q^{1}_{21}(t)dW_{1}(t)+q^{1}_{22}(t)dW_{2}(t)+\int_{\mathbb{R}_{0}}r^{1}_{2}(t-,z)\tilde{N}(dt,dz)+w_{2}^{1}(t)d\tilde{\Phi}(t), \nonumber\\
p_{2}^{1}(T)=&\ -2X_{2}(T),
\end{align}
where $w_{k}^{1}(t)d\tilde{\Phi}(t)=\sum_{j=1}^{N}w_{k}^{1,j}(t)d\tilde{\Phi}_{j}(t)$ for $k=1,2$ and $j=1,2,\dots, D$. \\
\\
Let us solve the Equations (\ref{p-1-1})-(\ref{p-1-2}).\\
Firstly, in order to find a solution for \ $p_{2}^{1}(t), \ t\in [0;T]$ \ let us try:
\begin{align*}
p_{2}^{1}(t)=&\ \phi(t,\alpha(t))X_{2}(t)\\
\phi(T,\alpha(T))=&\ \phi(T,e_{i})=-2, 
\end{align*}
where $\phi(\cdot, e_{i})$ is a $C^{1}$ deterministic function for all \ $e_{i}\in S$, \ $i=1,2,\ldots,D$ with the given terminal value.\\
We apply It\^{o}'s formula as described in \cite{zes:02}:
\begin{align}
\label{ito1}
dp_{2}^{1}(t)=&\ \biggl(\phi^{'}(t,\alpha(t-))X_{2}(t-)+\phi(t,\alpha(t-))X_{2}(t-)u(t) \nonumber \\
&+\sum_{j=1}^{N}X_{2}(t-)(\phi(t,e_{j})-\phi(t,\alpha(t-)))\mu_{j}(t)\biggr)dt \nonumber\\
&+\phi(t,\alpha(t-))X_{2}(t-)\sigma_{2}(t,\alpha(t-))dW_{2}(t) \nonumber\\
&+\int_{\mathbb{R}_{0}}\phi(t,\alpha(t-))X_{2}(t-)\eta(t-,\alpha(t-),z)\tilde{N}(dt,dz)\nonumber\\
&+\sum_{j=1}^{D}X_{2}(t-)(\phi(t,e_{j})-\phi(t,\alpha(t-)))d\tilde{\Phi}_{j}(t)
\end{align}
Now, we compare Equations (\ref{p-1-2})-(\ref{ito1}) and obtain the following solutions for $p_{2}^{1}(t), \ q^{1}_{21}(t), \ q^{1}_{22}(t), \ r_{2}^{1}(t-,z)$ \ and \ $w_{2}^{1}(t)$ \ for $t\in [0,T]$:
\begin{align*}
&q^{1}_{21}(t)=0,\\
&q^{1}_{22}(t)=\phi(t,e_{i})X_{2}(t-)\sigma_{2}(t,e_{i}),\\
&r_{2}^{1}(t-,z)=\phi(t,e_{i})X_{2}(t-)\eta(t-,\alpha(t-),z), \\
&w_{2}^{1,j}(t)=(\phi(t,e_{j})-\phi(t,e_{i}))X_{2}(t-),
\end{align*}
and
\begin{align*}
&-\phi(t,e_{i})X_{2}(t-)\biggl(u(t)+\sigma_{2}^{2}(t,e_{i})+\int_{\mathbb{R}_{0}}\eta^{2}(t-,e_{i},z)\nu(dz)\biggr) \\
&=X_{2}(t-)\biggl(\phi^{'}(t,e_{i})+\phi(t,e_{i})u(t)+\sum_{j=1}^{N}X_{2}(t-)(\phi(t,e_{j})-\phi(t,e_{i}))\mu_{ij}(t)\biggr).
\end{align*}
Hence,
\begin{align*}
&X_{2}(t-)\biggl[\phi^{'}(t,e_{i})+\phi(t,e_{i})\biggl(2u^{*}(t)+\sigma_{2}^{2}(t,e_{i})+\int_{\mathbb{R}_{0}}\eta^{2}(t-,e_{i},z)\nu(dz)\biggr)\\
&\quad +\sum_{j=1}^{D}(\phi(t,e_{j})-\phi(t,e_{i}))\mu_{ij}(t)\biggr]=0.
\end{align*}
Let us call
\begin{equation*}
B(t,e_{i})=2u^{*}(t)+\sigma_{2}^{2}(t,e_{i})+\int_{\mathbb{R}_{0}}\eta^{2}(t-,e_{i},z)\nu(dz).
\end{equation*}
Then, obviously, we get the following $N$-coupled differential equation with its terminal value as follows:
\begin{align*}
&\phi^{'}(t,e_{i})+\phi(t,e_{i})B(t,e_{i})+\sum_{j=1}^{D}(\phi(t,e_{j})-\phi(t,e_{i}))\mu_{ij}(t)=0, \\
&\Phi(T,e_{i})=-2, \quad \text{for} \quad i=1,2,\ldots,D.
\end{align*}
Finally, by applying Feyman-Kac procedure:
\begin{equation*}
\phi(t,e_{i})=\ -2E\biggl[\exp\biggl\{\int_{t}^{T}-B(t,e_{i})ds\biggr\}| \alpha(t-)=e_{i}\biggr], \ i=1,2,\ldots,D.
\end{equation*}
Moreover, we can find out $p_{1}^{1}(t), \ t\in [0,T]$ by trying $p_{1}^{1}(t)=g_{1}(t)$, where  $g(\cdot)$ is a deterministic function with terminal value $g(T)=\lambda_{1}$.\\
Then, by Equation (\ref{p-1-1}):  
\begin{equation*}
p_{1}^{1}(t)=\lambda_{1}, \quad q_{11}^{1}(t)=q_{12}^{1}(t)=r^{1}_{1}(t-,z)=w_{1}^{1}(t)=0.
\end{equation*}
Now, let us differentiate $H^{1}$ with respect to $\delta$ to define the optimal control process for the insurance company:
\begin{equation*}
\delta^{-\kappa_{1}}(t)h_{1}(t,\alpha(t-))-p^{1}_{1}=0
\end{equation*}
Then, 
\begin{equation*}
\delta(t)=\biggl(\frac{\lambda_{1}}{h_{1}(t,\alpha(t-))}\biggr)^{\frac{-1}{\kappa_{1}}}.
\end{equation*}
Finally, by applying expectation to both sides of the Equation (\ref{comp}) and by the constraint for Player 1, let us determine $\lambda_{1}$:
\begin{equation*}
\lambda_{1}=\ \biggl\{u-c-K_{1}+E\bigg[\int_{0}^{T}(\tilde{p}(t)-a(t,\alpha(t-)))dt\biggr]\biggr\}^{-\kappa_{1}}E\biggl[\int_{0}^{T}h_{1}^{\frac{1}{\kappa_{1}}}(t,\alpha(t-))dt\biggr]^{\kappa_{1}}
\end{equation*}
Now, let us represent the adjoint equations for Hamiltonian of the second player:
\begin{align}
\label{p-2-1}
dp_{1}^{2}(t)=&\ q^{2}_{11}(t)dW_{1}(t)+q^{2}_{12}(t)dW_{2}(t)+\int_{\mathbb{R}_{0}}r^{2}_{1}(t-,z)\tilde{N}(dt,dz)+w_{1}^{2}(t)d\tilde{\Phi}(t), \nonumber\\
p_{1}^{2}(T)=&\ \kappa_{2},
\end{align}
and
\begin{align}
\label{p-2-2}
dp_{2}^{2}(t)=&\ -\biggl(u(t)p^{2}_{2}(t)+\sigma_{2}(t,\alpha(t-))q^{2}_{22}(t)+\int_{\mathbb{R}_{0}}\eta(t-,e_{i},z)r^{2}_{2}(t-,z)\nu(dz)\biggr)dt\nonumber\\
&+q^{2}_{21}(t)dW_{1}(t)+q^{2}_{22}(t)dW_{2}(t)+\int_{\mathbb{R}_{0}}r^{2}_{2}(t-,z)\tilde{N}(dt,dz)+w_{2}^{2}(t)d\tilde{\Phi}(t),\nonumber\\
p_{2}^{2}(T)=&\ \frac{\lambda_{2}e^{-\tilde{r}(T,\alpha(T)}}{X_{2}(T)}, 
\end{align}
where $w_{k}^{2}(t)d\tilde{\Phi}(t)=\sum_{j=1}^{D}w_{k}^{2,j}(t)d\tilde{\Phi}_{j}(t)$ for $k=1,2$ and $j=1,2,\dots, D$.\\
\\
Now, let us solve these adjoint equations.\\
Let us try:
\begin{align*} 
&p_{2}^{2}(t)=\frac{A(t,\alpha(t))}{X_{2}(t)}, \qquad  \text{for} \quad t\in [0,T],\\
&A(T,\alpha(T))=A(T,e_{k})=\lambda_{2}e^{-\tilde{r}(T,e_{k})}, 
\end{align*}
where $A(\cdot,e_{k})$ is a deterministic $C^{1}$ function for all $k=1,2,\ldots,D$ with the given terminal value.\\
We apply It\^{o}'s formula as described in \cite{zes:02}:
\begin{align}
\label{ito2}
dp_{2}^{2}(t)=&\ \biggl[A^{'}(t,\alpha(t-))X_{2}^{-1}(t)+A(t,\alpha(t-))X_{2}^{-1}(t)\biggl(-u(t)+\sigma^{2}_{2}(t,\alpha(t-)) \nonumber\\
&+\int_{\mathbb{R}_{0}}\left\{(1+\eta(t-,\alpha(t-),z))^{-1}-1+\eta(t-,\alpha(t-),z)\right\}\nu(dz)\biggr)\nonumber\\
&+ \sum_{j=1}^{N}X_{2}^{-1}(t)\left\{A(t,e_{j})-A(t,\alpha(t-))\right\}\mu_{j}(t)\biggr]dt\nonumber\\
&+A(t,\alpha(t-))X_{2}^{-1}(t)\biggl[-\sigma_{2}(t,\alpha(t-))dW_{2}(t) \nonumber \\
& +\int_{\mathbb{R}_{0}}\left\{(1+\eta(t-,\alpha(t-),z))^{-1}-1\right\}\tilde{N}(dt,dz)\biggr]\nonumber\\
&+\sum_{j=1}^{D}X_{2}^{-1}(t)\left\{A(t,e_{j})-A(t,\alpha(t-))\right\}d\tilde{\Phi}_{j}(t)
\end{align}
Now, we compare the Equations  (\ref{p-2-2}) and (\ref{ito2}), we obtain:
\begin{align}
\label{replace}
&-\biggl(u(t)p_{2}^{2}(t)+\sigma_{2}(t,\alpha(t-))q^{2}_{22}(t)+\int_{\mathbb{R}_{0}}\eta(t-,\alpha(t-),z))r^{2}_{2}(t-,z)\nu(dz)\biggr)\nonumber \\
&=A^{'}(t,\alpha(t-))X_{2}^{-1}(t)+A(t,\alpha(t-))X_{2}^{-1}(t)\biggl(-u(t)+\sigma^{2}_{2}(t,\alpha(t-)) \nonumber\\
&+\int_{\mathbb{R}_{0}}\left\{(1+\eta(t-,\alpha(t-),z))^{-1}-1+\eta(t-,\alpha(t-),z)\right\}\nu(dz)\biggr)\nonumber\\
&+ \sum_{j=1}^{D}X_{2}^{-1}(t)\left\{A(t,e_{j})-A(t,\alpha(t-))\right\}\mu_{j}(t)
\end{align}
and 
\begin{align*}
&q^{2}_{21}(t)=0, \\
&q^{2}_{22}(t)=-A(t,e_{i})X_{2}^{-1}\sigma_{2}(t,e_{i}), \\
&r^{2}_{2}(t)=A(t,e_{i})X_{2}^{-1}(t)\biggl((1+\eta(t-,e_{i},z))^{-1}-1\biggr), \\
&w_{2}^{2,j}(t)=X_{2}^{-1}\left\{A(t,e_{j})-A(t,e_{i})\right\}, \ \text{for} \ i=1,2,\ldots,D.
\end{align*}
If we replace the values of $p_{2}^{2},\ q^{2}_{21}$, and $r^{2}_{2}$ values in Equation (\ref{replace}), the we get:
\begin{align*}
&A^{'}(t,e_{i})X_{2}^{-1}(t)+\int_{\mathbb{R}_{0}}A(t,e_{i})X_{2}^{-1}(t)\biggl(\frac{\eta(t-,e_{i},z)}{\eta(t-,e_{i},z)+1}+\frac{1}{\eta(t-,e_{i},z)+1}-1\biggl)\nu(dz)\\
&\quad +\sum_{j=1}^{D}X_{2}^{-1}(t)\left\{A(t,e_{j})-A(t,e_{i})\right\}\mu_{ij}(t)=0
\end{align*}
Hence, finally, we have:
\begin{equation*}
X_{2}^{-1}\biggl[A^{'}(t,e_{i})+\sum_{j=1}^{D}\left\{A(t,e_{j})-A(t,e_{i})\right\}\mu_{ij}(t)\biggr]=0.
\end{equation*}
Then, 
\begin{align*}
&A^{'}(t,e_{i})+\sum_{j=1}^{D}\left\{A(t,e_{j})-A(t,e_{i})\right\}\mu_{ij}(t)=0,\\
&A(T,\alpha(T))=A(T,e_{k})=\lambda_{2}e^{-\tilde{r}(T,e_{k})}
\end{align*}
for any $e_{k}\in S, \ k=1,2,\ldots,D$.\\
By applying classical Feyman-Kac procedure, we can solve these $N$-coupled equations:
\begin{align*}
A(t,\alpha(t))=&\ \lambda_{2}E\biggl[e^{-\tilde{r}(T,e_{k})}| \alpha(t-)=e_{i}\biggr]\\
=&\lambda_{2}e^{-\tilde{r}(T,\alpha(T))}, \ \text{for any} \ t\in [0,T].
\end{align*}
Now, let us find $p_{1}^{2}(t), \ t\in [0,T]$ by trying $p_{1}^{2}(t)=h(t)$, where  $h(t)$ is a deterministic function with terminal value $h(T)=\kappa_{2}$.\\
Then, by Equation (\ref{p-2-1}):  
\begin{equation*}
p_{1}^{2}(t)=\kappa_{2}, \quad q_{11}^{2}(t)=q_{12}^{2}(t)=r^{2}_{1}(t-,z)=w_{1}^{2}(t)=0.
\end{equation*}
Let us differentiate $H^{2}$ with respect to $u$ to define the optimal control process for the bank:
\begin{equation*}
\frac{h_{2}(t,\alpha(t-))}{u(t)}+X_{2}(t)p_{2}^{2}(t)=0
\end{equation*}
Then,
\begin{equation*}
u(t)=\frac{-1}{\lambda_{2}}e^{\tilde{r}(T,\alpha(T))}h_{2}(t,\alpha(t-)), \ t\in [0,T].
\end{equation*}
In order to determine $\lambda_{2}$, let us apply It\^{o}'s formula to $Y(t)=\ln(X_{2}(t))$:
\begin{align*}
dY(t)&=\ \biggl\{\frac{-1}{\lambda_{2}}e^{\tilde{r}(T,\alpha(T))}h_{2}(t,\alpha(t-))-\frac{1}{2}\sigma_{2}^{2}(t,\alpha(t-))\\
&+\int_{\mathbb{R}_{0}}\biggl(\ln(\eta(t-,\alpha(t-),z)+1)-\eta(t-,\alpha(t-),z)\biggr)\nu(dz)\biggr\}dt\\
&+\sigma_{2}(t,\alpha(t-))dW_{2}(t)+\int_{\mathbb{R}_{0}}\ln(\eta(t-,\alpha(t-),z)+1)\tilde{N}(dt,dz).
\end{align*}
If we multiply both sides of the equation by $e^{-\tilde{r}(T,\alpha(T))}$ and apply expectation, we get:
\begin{align*}
&E[e^{-\tilde{r}(T,\alpha(T))}\ln(X_{2}(T))]=E\biggl[e^{-\tilde{r}(T,\alpha(T))}\ln(c)+\int_{0}^{T}\biggl\{ \frac{-1}{\lambda_{2}}h_{2}(t,\alpha(t-))-\frac{1}{2}e^{-\tilde{r}(T,\alpha(T))}\\
&\times \sigma_{2}^{2}(t,\alpha(t-))+\int_{\mathbb{R}_{0}}e^{-\tilde{r}(T,\alpha(T))}\biggl(\ln(\eta(t-,\alpha(t-),z)+1)-\eta(t-,\alpha(t-),z)\biggr)\nu(dz)\biggr\}dt\biggr].
\end{align*}
Finally, let us call:
\begin{align*}
D_{1}&=\ e^{-\tilde{r}(T,\alpha(T))}\ln(c),\\
D_{2}&=\ E\biggl[\int_{0}^{T}h_{2}(t,\alpha(t-))dt\biggr],\\
D_{3}&=\ E\biggl[\int_{0}^{T}e^{-\tilde{r}(T,\alpha(T))}\biggl(-\frac{1}{2}\sigma_{2}^{2}(t,\alpha(t-))\\
&\quad +\int_{\mathbb{R}_{0}}\biggl\{\ln(\eta(t-,\alpha(t-),z)+1)-\eta(t-,\alpha(t-),z)\biggr\}\nu(dz)\biggr)dt\biggr].
\end{align*}
Therefore, by the constraint for Player 2, we select $\lambda_{2}$ such that:
\begin{equation*}
\lambda_{2}=\frac{D_{2}}{D_{1}+D_{3}-K_{2}}>0.
\end{equation*}
Finally, by measurability and square-integrability conditions for $\sigma_{k}, \ \eta, \ \gamma, \ h_{k}$ \ and selection of \ $g_{k}^{\lambda_{k}}(y,e_{i})=g_{k}(y,e_{i})+\lambda_{k}(M_{k}(y,e_{i}))$, \ for \ $i=1,2,\ldots,D$ \ and \ $k=1,2$, \ one can easily verify the integrability and concavity conditions of Theorem \ref{suff}.


\section{Conclusion}
\label{conc}

In this work, we developed techniques to solve stochastic optimal control problems in a Lagrangian game theoretical environment. Both of the zero-sum and nonzero-sum stochastic differential game problems with two specific type of constraints can be approached by dynamic programing principle and stochastic maximum principle within the construction of our theorems. Moreover, we demonstrated these theorems for a quite extended model of stochastic processes, named Markov regime-swithcing jump-diffusions. As we explained in Section \ref{sec:intro}, such models have a wide range of application area. In our work, we focused on a business agreement, called Bancassurance, between a bank and an insurance company by the methods of stochastic maximum principle for a nonzero-sum stochastic differential game. We investigated optimal dividend strategy for the company as a best response according to the optimal mean rate of return choice of a bank for its own cash flow and vice versa. We found out a Nash equilibrium for this game and solved the adjoint equations explicitly for each state. \\
It is well known that the timing and the amount of dividend payments are strategic decisions for companies. The announcement of a dividend payment may reduce or increase the stock prices of a company. A high dividend payment may give a message to shareholders and potential investors about substantial amount of profits achieved by the company. On the other side, it may create an impression of that the company does not have a good future project to invest in rather than paying to investors. Moreover, dividend payments may aim to honor the shareholders feeling of getting a reward for their trust in the company.\\
From the side of the bank, it is clear that creating a cash flow with high returns would be the main goal. It is obviously seen that depending on the values of $h_{2}(\cdot,e_{i}), \ e_{i}\in S, \ i=1,2,\ldots,D$, \ the appreciation rate of the bank's investment may drop below zero.  \\
Hence, in our formulation, we provide an insight to both of the bank and the insurance company about their best moves in a bancassurance commitment under specified technical conditions. 

\section*{Disclosure statement}

No conflict of interest.

\section*{Funding}

This project is supported by SCROLLER:A Stochastic ContROL approach to Machine Learning with applications to Environmental Risk models, Project 299897 from the Norwegian Research Council.


\bibliographystyle{tfnlm}
\bibliography{biblio}

\section*{Appendix}

Let us clarify the general formulation of the technique that we apply here for the solution of a nonzero-sum stochastic differential game within this context of Equations (\ref{eq:3.1})-(\ref{3.1.1}) and the problems (\ref{j1})-(\ref{j2}) by stochastic maximum principle for a Markov regime-switching jump-diffusion model:\\
The Hamiltonian functions associated with Player $k$, namely $H_{k}$, for $k=1,2$, defined from $[0,T]\times \mathbb{R}^{N}\times U_{1}\times U_{2}\times \mathbb{R}^{N}\times \mathbb{R}^{N\times M}\times \mathcal{R}\times \mathbb{R}^{N\times D} \times S\times$ to  $\mathbb{R}$ as follows:
\begin{align*}
&H^{k}(t,y,u_{1},u_{2},p^{k},q^{k},r^{k}(\cdot),w^{k},e_{i})=\ f_{k}(t,y,u_{1},u_{2},e_{i})+b^{T}(t,y,u_{1},u_{2},e_{i})p^{k}\\
&\qquad+tr(\sigma^{T}(t,y,u_{1},u_{2},e_{i})q^{k})+\int_{\mathbb{R}_{N}}\sum_{l=1}^{L}\sum_{n=1}^{N}\eta_{nl}(t,y,u_{1},u_{2},e_{i},z)r_{nl}^{k}(t,z)\nu_{l}(dz)\\
&\qquad+\sum_{j=1}^{D}\sum_{n=1}^{N}\gamma_{nj}(t,y,u_{1},u_{2},e_{i},z)w_{nj}^{k}(t)\mu_{ij}, \qquad k=1,2,
\end{align*}
and each $H^{k}, \ k=1,2$, is continuously differentiable with respect to $y$; i.e., each is a $C^{1}$-function with respect to $y$, and differentiable with respect to corresponding Player's control processes.\\
Corresponding adjoint equations for Player $k$, \ for $k=1,2$, \ in the unknown adapted processes $p^{k}(t)\in \mathbb{R}^{N}$, \ $q^{k}(t)\in \mathbb{R}^{N\times M}$, \ $r^{k}(t-,z)\in \mathcal{R}$, where $\mathcal{R}$ is the set of functions $r:[0,T]\times \mathbb{R}_{0}\rightarrow \mathbb{R}^{N\times L}$, and $w^{k}(t)\in \mathbb{R}^{N\times D}$ are given by the following equations:
\begin{align}
dp^{k}(t)=&\ -\nabla_{y}H_{k}(t,Y(t),u_{1}(t),u_{2}(t),p^{k}(t),q^{k}(t),r^{k}(t,\cdot),w^{k}(t),\alpha(t))dt \nonumber\\
&\quad +q^{k}(t)dW(t)+\int_{\mathbb{R}_{0}}r^{k}(t-,z)\tilde{N}(dt,dz), \qquad t<T,  \label{adj1}\\
p^{k}(T)=&\ \nabla g_{k}(Y(T),\alpha(T)), \qquad k=1,2, \label{adj2}
\end{align}
where $\nabla_{y} \phi(\cdot)=(\frac{\partial\phi}{\partial y_{1}},\ldots,\frac{\partial\phi}{\partial y_{N}})^{T}$ is the gradient of $\phi:\mathbb{R}^{N}\rightarrow \mathbb{R}$ with respect to $y=(y_{1},\ldots,y_{N})$.
For the existence–uniqueness results of the BSDEs with jumps and regimes (\ref{adj1})-(\ref{adj2}), see Propositions 5.1 and 5.2 by Crépey and Matoussi \cite{crepey}. In this context, here, we assume that $p^{k}(t),\ q^{k}(t), \ r^{k}(t-,z)$, \ and \ $w^{k}(t), \ k=1,2$ are square-integrable.\\
Now, we can present a sufficient maximum principle for such a game:
\begin{theorem}
\label{suff}
Let $(u_{1}^{*},u_{2}^{*})\in \Theta_{1}\times \Theta_{2}$ with a corresponding solution $\hat{Y}(t):=Y^{u_{1}^{*},u_{2}^{*}}(t)$ and suppose there exists an adapted solution  $(p^{k}(t),q^{k}(t),r^{k}(t-,z),w^{k}(t)), \ k=1,2$, of the corresponding adjoint equations (\ref{adj1})-(\ref{adj2}) such that for all $(u_{1},u_{2})\in \Theta_{1}\times \Theta_{2}$, we have:
\begin{align*}
&E\biggl[\int_{0}^{T}(\hat{Y}(t)-Y^{u_{1}}(t))^{T}\biggl\{ \hat{q}^{1}(t)\hat{q}^{1}(t)^{T}+\int_{\mathbb{R}_{0}}\hat{r}^{1}(t-,z)\hat{r}^{1}(t-,z)^{T}\nu(dz)\\
&\qquad \qquad+\hat{w}^{1}(t)Diag(\mu(t))\hat{w}^{1}(t)^{T}\biggr\}(\hat{Y}(t)-Y^{u_{1}}(t))^{T}dt\biggr]< \infty,
\end{align*}
and
\begin{align*}
&E\biggl[\int_{0}^{T}(\hat{Y}(t)-Y^{u_{2}}(t))^{T}\biggl\{ \hat{q}^{2}(t)\hat{q}^{2}(t)^{T}+\int_{\mathbb{R}_{0}}\hat{r}^{2}(t-,z)\hat{r}^{2}(t-,z)^{T}\nu(dz)\\
&\qquad \qquad+\hat{w}^{2}(t)Diag(\mu(t))\hat{w}^{2}(t)^{T}\biggr\}(\hat{Y}(t)-Y^{u_{2}}(t))^{T}dt\biggr]< \infty,
\end{align*}
where $Y^{u_{1}}(t):=Y^{u_{1},u^{*}_{2}}(t)$ and $Y^{u_{2}}(t):=Y^{u^{*}_{1},u_{2}}(t)$.\\
Furthermore,
\begin{align*}
&E\biggl[\int_{0}^{T}\hat{p}^{1}(t)^{T}\biggl((\sigma(t,Y^{u_{1}}(t),\alpha(t),u_{1}(t),u_{2}^{*}(t))-\hat{\sigma}(t,\hat{Y}(t),\alpha(t),u_{1}^{*}(t),u^{*}_{2}(t)))^{2}\\
&\quad+\int_{\mathbb{R}_{0}}(\eta(t,Y^{u_{1}}(t),\alpha(t),u_{1}(t),u_{2}^{*}(t),z)-\hat{\eta}(t,\hat{Y}(t),\alpha(t),u_{1}^{*}(t),u^{*}_{2}(t),z))^{2}\nu(dz) \\ 
&\quad+\sum_{j=1}^{D}(\gamma^{j}(t,Y^{u_{1}}(t),\alpha(t),u_{1}(t),u_{2}^{*}(t))-\hat{\gamma}^{j}(t,\hat{Y}(t),\alpha(t),u_{1}^{*}(t),u^{*}_{2}(t)))^{2}\lambda_{j}(t)\biggr)\hat{p}^{1}(t) dt\biggr]<\infty
\end{align*}
and
\begin{align*}
&E\biggl[\int_{0}^{T}\hat{p}^{2}(t)^{T}\biggl((\sigma(t,Y^{u_{2}}(t),\alpha(t),u_{1}^{*}(t),u_{2}(t))-\hat{\sigma}(t,\hat{Y}(t),\alpha(t),u_{1}^{*}(t),u_{2}^{*}(t)))^{2}\\
&\quad+\int_{\mathbb{R}_{0}}(\eta(t,Y^{u_{2}}(t),\alpha(t),u_{1}^{*}(t),u_{2}(t),z)-\hat{\eta}(t,\hat{Y}(t),\alpha(t),u_{1}^{*}(t),u_{2}^{*}(t),z))^{2}\nu(dz) \\ 
&\quad+\sum_{j=1}^{D}(\gamma^{j}(t,Y^{u_{2}}(t),\alpha(t),u_{1}^{*}(t),u_{2}(t))-\hat{\gamma}^{j}(t,\hat{Y}(t),\alpha(t),u_{1}^{*}(t),u^{*}_{2}(t)))^{2}\lambda_{j}(t)\biggr)\hat{p}^{2}(t) dt\biggr]<\infty.
\end{align*} 
Moreover, assume that the following conditions hold:
\begin{enumerate}
\item For almost all $t\in [0,T]$,
\begin{align*}
&H^{1}(t,\hat{Y}(t-),u^{*}_{1}(t),u^{*}_{2}(t),\hat{p}^{1}(t),\hat{q}^{1}(t),\hat{r}^{1}(t,\cdot),\hat{w}^{1}(t),\alpha(t-))\\
&\quad=\sup_{u_{1}\in U_{1}}H^{1}(t,\hat{Y}(t-),u_{1}(t),u^{*}_{2}(t),\hat{p}^{1}(t),\hat{q}^{1}(t),\hat{r}^{1}(t,\cdot),\hat{w}^{1}(t),\alpha(t-)),
\end{align*}
and 
\begin{align*}
&H^{2}(t,\hat{Y}(t-),u^{*}_{1}(t),u^{*}_{2}(t),\hat{p}^{2}(t),\hat{q}^{2}(t),\hat{r}^{2}(t,\cdot),\hat{w}^{2}(t),\alpha(t-))\\
&\quad=\sup_{u_{2}\in U_{2}}H^{2}(t,\hat{Y}(t-),u_{1}^{*}(t),u_{2}(t),\hat{p}^{2}(t),\hat{q}^{2}(t),\hat{r}^{2}(t,\cdot),\hat{w}^{2}(t),\alpha(t-)).
\end{align*}
\item For each fixed pair of $(t,e_{i})\in [0,T]\times S$,
\begin{equation*}
\hat{H^{1}}(y)=\sup_{u_{1}\in U_{1}}H^{1}(t,y,u_{1},u_{2}^{*}(t),\hat{p}^{1}(t),\hat{q}^{1}(t),\hat{r}^{1}(t,\cdot),\hat{w}^{1}(t),e_{i}),
\end{equation*}
and 
\begin{equation*}
\hat{H^{2}}(y)=\sup_{u_{2}\in U_{2}}H^{2}(t,y,u_{1}^{*}(t),u_{2},\hat{p}^{2}(t),\hat{q}^{2}(t),\hat{r}^{2}(t,\cdot),\hat{w}^{2}(t),e_{i})
\end{equation*}
exist and are concave functions of $y$.      
\item $g_{k}(y,e_{i}), \ k=1,2$,  are concave functions of $y$ for each $e_{i}\in S$.
\end{enumerate}
Then, $(u_{1}^{*},u_{2}^{*})\in \Theta_{1}\times \Theta_{2}$ is a Nash equilibrium of the system (\ref{eq:3.1})-(\ref{3.1.1}) and the problems (\ref{j1})-(\ref{j2}).
\end{theorem}
\begin{proof}
For the proof of this theorem, it is enough to follow the steps of Theorem 3.1 in \cite{zes:02} in our game theoretical formulation for each player. Moreover, the proof may be seen as a special version of Thoerem 3.1 in \cite{pamen}.
\end{proof}

\end{document}